\documentclass[a4,12pt]{amsart}

\usepackage[T2A]{fontenc}
\usepackage[utf8]{inputenc}
\usepackage[english]{babel}

\usepackage{hyperref}
\usepackage{amsmath}
\usepackage{amssymb}
\usepackage{amscd}
\usepackage{amsfonts,amsmath,amsthm}

\usepackage[matrix,arrow,curve]{xy}
\usepackage{xfrac}
\usepackage{color}

\DeclareMathOperator{\sq}{\mathbb{D}}

\newtheorem{theorem}{Theorem}
\newtheorem{prop}{Proposition}
\newtheorem*{cor}{Corollary}
\newtheorem{lemma}{Lemma}
\newtheorem*{main lemma}{The main lemma}

\newtheorem{definition}{Definition}

\newcommand{\gm}{{\Small GM}}

\title{Critical point Theory for the Distance Function on PL-Spaces}

\author{ Dirk Siersma}

\address{
 D. Siersma: Mathematisch Instituut, Universiteit Utrecht  d.siersma@uu.nl }
\date{\today}

\begin{document}

\begin{abstract}
We  study the (typical) singularities of the distance function from a point in the ambient space to a polytopal complex in Euclidean space. By means of the Goresky-MacPherson splitting we generalize the existing linear theory to the quadratic distance function.
Critical points occur on `selected' strata of the PL-space, where normal Morse Data determine the change of topology.
On the homology level there exist  indices with multiplicities, combined in a Morse vector. We focus mainly on PL-manifolds. We give bounds, depending on the dimension of the strata. In low dimensions we describe singularity types. At the end we discuss relations with discrete Morse Theory.

\end{abstract}

\maketitle

\section{Introduction}\label{SectIntro}

Critical point theory and especially Morse Theory  form a bridge between geometric topology  and analysis.
Marston Morse started this theory in the first half of the 19th century.
A corner stone is his book: \textit{Critical Point Theory in Global analysis and Differentiable Topology (1969)}(\cite{MC}).
The theory has been developed in the finite dimensional and in the infinite dimensional case. While first formulated for smooth manifolds and differential functions, the approach was extended to singular spaces and non-differentiable functions.

 In this article we study (embedded) polyhedral spaces  $W  \subset \mathbb{R}^n$, with a differentiable function on the ambient space, especially   the  squared distance function  $\sq_u$ from a punt $u \in \mathbb{R}^N$.  In the $C^2$-manifold  case it is well known that    critical points of $\sq_u$ correspond to normals to $M$ emanating from $u$ and via versa. In the PL-manifold  case this is also true, with appropriate definitions.
This is related  to the concept of  Euclidean distance degree, mostly studied in the context of algebraic sets
\cite{DHOST}. A PL-approach seems to be of interest.

In the case of a singular (=non-smooth) space the study \cite{GM}  of Goresky-Macpherson (\gm) applies. A first issue is to define the concept of critical point in relation to a stratification of the (singular) space. A GM-critical point is defined as a critical point restricted to a stratum. They describe  at  GM-critical points the local type by Morse-data, which represent the behavior of the function on the link of a singular point. 

Around 1970  Banchoff \cite{Banch}, \cite{Ba3}, and  Kuiper \cite{Ku} studied linear functions on on PL-manifolds and the relation with  curvature. More recently there is the work of Grunet, Kühnel and Rote  \cite{GKR}.  This treats  critical points of PL-functions on PL-spaces (especially PL-manifolds). They deal with the standard  statements of Morse Theory, such as regular interval theorem, cell-attaching when passing a critical level and generalized Morse inequalities. An interesting observation is, that in the PL-manifold case it is generically possible that several cells attach around a single critical point, also cells from different dimension. 

The same statements occur in the GM-approach, which considers more general spaces. We combine both approaches to study the quadratic distance function on a PL-space. Important is that near a
 GM-critical point the function is locally equivalent to a linear function, determined by the derivative.\\
The local Morse Data is a product of Normal Morse Data and Traversal Morse Data. The restriction of the squared distance to a normal slice is therefore a linear function  on a PL-space.  We can use the results from \cite{GKR}  and can combine them to the local Morse Data.  

 Critical points occur on `selected' strata of the PL-space, where normal Morse Data determine the change of topology.
On the homology level there exist  indices with multiplicities, combined in a Morse vector. We give bounds, depending on the dimension of the strata. In low dimensions we describe the  singularity types.

We also mention that the Morse-inequalities are valid in the GM-approach. This was not explicitly  written, but known to the experts.

We restrict from the second part of Section \ref{s:smt} to PL-manifolds. In the case of low dimensions we have a precise picture of  the attaching around the critical level. At the homology level we encode this into a Morse vector.

We remark that not every \gm-critical point is a topological critical point. Sometimes (very often) GM-critical points are topologically regular, i.e. the function is locally topological equivalent to a linear function. Moreover trivial  homotopy and homology does not imply topological regularity (examples and discussion are in \cite{GKR}).

The cases of polytopes in low dimension (including PL-curves)  have been studied in connection  to count the (maximum) number of normals from a point to the PL-space (i.e. the critical points of $\sq_u$. In section \ref{s:count} we summarize the results and give the connection the  Concurrent Normal Conjecture for convex bodies.

A natural question is the relation to Discrete Morse Theory. For linear functions on a PL-manifold   and in low dimensions this has been confirmed in   \cite{FLV}. In Section \ref{s:dmt} we show that also a correspondence for the quadratic distance function should exist and we sketch two approaches to confirm this.

The work on critical points of the distance function and the connections with normals started at Research in Pairs at CIRM Luminy 2024 and were continued during the Winter Math Weeks 2025 and
2026 in Batumi (Georgia). We thank the hosts for the good working atmosphere and moreover  George Khimshiashvili and Gaiane Panina for discussion and useful remarks.

%%%%%%%%%%%%%%%
\section{PL-spaces and PL-functions}\label{s:PL}
%%%%%%%%%%%%%%%
\subsection{Polytopal complexes, PL manifolds and PL functions}\label{ss:Pl-function}

We will recall first the definition of (embedded)  polytopal complex and PL-manifold $M$. 

\begin{definition} A subset of the Euclidean space $\mathbb{R}^N$ is called
a \textbf{polyhedral set}, if it is the intersection of finitely many closed half-spaces of $\mathbb{R}^N$.
Bounded polyhedral sets are called convex  polytopes.
\end{definition}
A polyhedral set has a natural stratification into linear cells, also called faces, which are convex, non-void interior in its affine span and have a well defined dimension.

\begin{definition}
A \textbf{polytopal complex}, or \textbf{PL-space} is a finite family
\textbf{convex polytopes}  in some   $\mathbb{R}^N$ such that the following closure properties hold:
\begin{itemize} 
\item Each face of a family member is again a family member.
\item The intersection of two family members is a face of both members, possibly
the empty set, and thus a family member by the first property.
\end{itemize}
The family members are called the cells of the complex. A \textbf{simplicial
complex} is a polytopal complex whose cells are simplices.
\end{definition}
The star $st(\sigma)$ of a cell $\sigma$ consists of those cells, whose closure contains $\sigma$. The link $lk (\sigma)$ are the cells in  $\overline{st(\sigma)}\setminus st(\sigma) $.
\

\begin{definition}
A \textbf{ PL-manifold}  $M$ is a finite polytopal complex $K$, build up from convex $d$-polytopes such that $|K|$ is homeomorphic to $M$ and such that the star of each (relatively open) cell is piecewise linear homeomorphic to an open ball in $d$-space.
\end{definition}

\noindent
We will assume that all our polyhedral complexes are embedded in $\mathbb{R}^N$.

\subsection{Facts about PL-functions}
We next define  (recall)  definitions and statements about PL-functions.
Our main source is \cite{GKR}, which is written  for abstract PL-manifolds. We will apply it to linear functions for an embedded PL-space.

\begin{definition}
A point $p \in M$ is called a (topologically) regular point (short: topo-regular) of a function $f$ if $f$ is at $p$ locally topologically equivalent to a non-vanishing linear function. The point $p$ is singular if its is not regular.
\end{definition}

\begin{definition}[3.1 in \cite{GKR}(generic PL-function)]
Let P be a finite polytopal complex P. A PL-function $f : P  \to \mathbf{R}$ is called generic PL if $f(v) \ne f(w)$ for any two distinct vertices $v,w$ of $P$.
\end{definition}

\begin{lemma}[3.2 in\cite{GKR}]
The PL-general interval theorem holds  if $f^{-1}[a,b]$ contains no vertex of $P$.
\end{lemma}

For a change in topology it remains to look at the vertices:

\begin{definition}[3.3 in \cite{GKR}] \label{d:hcrit} (H-critical point). Let $f$ be generic PL and $v$ a vertex with level $f(v) = a.$ Then $v$ is called H-critical if $H_{*}(f_a,f_a \setminus {v}) \ne O$.
Coefficients are taken in a field. 
The rank of $H_{k}(f_a,f_a \setminus {v})$ is called the multiplicity of index $k$ at $v$. The sum of all  multiplicities is called the total multiplicity. If the total multiplicity is greater than or equal to $1$ we say $v$ is an H-critical point.

\end{definition}

\begin{theorem}[3.5 in \cite{GKR}](Morse relations, duality)
For  a generic PL-function on a compact {\bf PL-d-manifold} we have  Morse inequalities and duality, where we include multiplicities and multiple indices.
\end{theorem}

%%%%%%%%%%%%%%%%%%%%%%%%%%%%%%
\section{Stratified Morse theory}\label{s:smt}
\subsection{Introduction}
Goresky and MacPherson \cite{GM} developed a critical point theory  for a smooth function $f: W \to \mathbb{R}$ on a stratified space $W$. See also  the short version \cite{Go} of Goresky.  
We will apply this below to our situation:  
$$\mathbb{D}_u (x) = | u -x |^2 \; \; ; \; x\in W \; ; \; u \in \mathbb{R}^N$$
  on a polytopal complex or a PL-manifold. 
Stratified Morse Theory is a very general; it applies to Whitney stratified spaces. In the PL-case the theory simplifies; many assumptions and constructions in \cite{GM} are satisfied or clear.  
The concept of critical point of a function $f$ is defined  in \cite{Go} as a critical point of $f$ restricted to a stratum. We will call this a\textit{ \gm -critical point}.  It is more general than what we will consider as critical point later). As an example every vertex is a   \gm -critical point.

{\bf Notation.} With $W_{\le c}$ we mean the set of all points of $W$, where the value of $f$ is $\le c$. Other notations: $W_{[ab]} = f^{-1}[a,b]$  and $W_b = f^{-1}(b)$.

\begin{theorem} [5.5.1 in \cite{Go}] Suppose the interval $[a,b]$ contains no  \gm -critical values of $f$. Then $W_{\le a}$ is homeomorphic to $W_{\le b}$, preserving the stratification.
\end{theorem}

\begin{definition}[2.3 in \cite{GM}]
A \gm-critical point $p \in W$ of the function $f: W \to \mathbf{R}$ is Morse  provided:
\begin{itemize}
\item its restriction to $W$ is proper
\item the critical point is isolated,
\item the restriction of $f$ to the  each stratum has a Morse  ($=$ non-degenerate) critical point at $p$,
\item  at each critical point on stratum $X$ the cotangent vector $df(p)$ is non-degenerate for $W$.
\end{itemize}
\end{definition}

A cotangent vector $\xi$ is conormal at $p$ if $\xi (T_pS) = 0$ if $S$ is the stratum containing $p$.  In the PL-case the definition of degenerate covector $\xi$ reduces to: There exist a face $S^{\star}$, adjacent to $p$ such that $\xi (Span (S^{\star})) = 0$.

\begin{definition}
The local Morse Data for $f$ at a \gm-critical point $p$ with value $f(p)=c$ is defined by the (topological) pair $(W_{[c-\epsilon,c +\delta]}, W_{c-\epsilon}) \cap  B_{\delta} $  where    $B_{\delta}$ is a small enough ball around $p$.
\end{definition}

\begin{theorem}[3.2. in \cite{GM}]\label{t:attaching}
If $c$ is the only  \gm -critical value between $c - \epsilon$ and $c + \epsilon$ then $W_{\le c +\epsilon}$ is homotopy equivalent to
$W_{\le c - \epsilon}$ plus attaching via the local Morse Data: $W_{[c-\epsilon,c +\epsilon]} \cap  B_{\delta}$ along the space  $W_{c-\epsilon} \cap  B_{\delta} $ .
\end{theorem}

So the change in topology is determined by the local Morse Data. These data can be computed by a splitting in tangential and normal direction via
the main theorem of \cite{GM}:

\begin{theorem}[3.5.4 in \cite{GM}]\label{t:data-product}
For a fixed stratification and a fixed function $f$ with Morse 
 \gm -critical points $p\in W$ there is a  homeomorphism:
$$ \mbox{\rm Local Morse Data} \cong \mbox{\rm Tangential Morse Data} \times \mbox{\rm Normal Morse Data} $$
\end{theorem}
The local Normal Morse Data is defined as  the local Morse Data of $f$ restricted to a normal hyperplane $N_p$ at $p$ and the Tangential Morse Data is the local Morse Data of $f$, restricted to the stratum, which contains $p$. \\ 
We will use notations:
$N_{\delta}(p) = B_{\delta}(p) \cap N_p$ and $\partial N_{\delta}(p) = \partial B_{\delta}(p) \cap N_p$.
 \\ NB. The product 
$(A_1,B_1) \times (A_2,B_2) =(A_1 \times A_2 , A_1 \times B_2 \cup B_1 \times A_2)$.
As a corollary the attaching in Theorem \ref{t:attaching} space is 
$$(D^{\lambda},\partial D^{\lambda}) \times D^{s - \lambda} \times (N_{\delta}(p) \cap W_{[v-\epsilon,v +\epsilon]},(N_{\delta}(p)\cap W_{v-\epsilon}). $$
where $ s$ is dimension stratum and
where $\lambda$ is the index of the Morse critical point of $f$ restricted to the stratum.
\smallskip

\subsection{Normal Morse Data depends only on the co-vector}\; \\
This follows from:\\
{\bf Proposition 3.6.2 (\cite{GM}}) 
{\em The (stratified) homeomorphism type of tangential and normal Morse Data are well defined (i.e. independent of the choices, which are made in their construction).}\\

The important corollary is that the normal Morse Data depend only on $df(p) $.  This means that we can replace locally the function $f$ by its linear part. See also Theorem 5.5.2 from \cite{Go}. In the PL-context, this means that not only the normal slice is PL (again a polytopal complex), but also the function on it can supposed to be linear.

%%%%%%%%%%%%%%%%%
\subsection{Homotopy description and half link}

Following \cite{GM} we define the upper half link as the pair of spaces
$$( L^{+},\partial L^{+})= (N_{\delta}(p)\cap W_{c+\epsilon},   \partial N_{\delta}(p)\cap W_{c+\epsilon}) $$
and the lower half link
$$(L^{-},\partial   L^{- } )= (N_{\delta}(p)\cap W_{c-\epsilon},   \partial N_{\delta}(p)\cap W_{c- \epsilon}) $$
Note that $\partial N_{\delta}(p)\cap W_{c+\epsilon}$ and $\partial N_{\delta}(p)\cap W_{c-\epsilon})$  are naturally  homeomorphic and (by gluing along  $\partial L^{-} \approx \partial L^{+}:= \partial L^{\pm}$) the (abstract) union  $L^{+} \cup_{\partial L^{\pm}} L^{-}$ makes sense and is homeomorphic to the link of $p$.
We recall from \cite{GM}:  
\begin{theorem}(3.11.1) The normal Morse Data is homeomorphic to the pair of spaces ($\mbox{\rm cone}(L^{+} \cup_{\partial L} L^{-}), L^{-})$.
\end{theorem}

\begin{cor}(3.11.2)
The normal Morse Data has the homotopy type to the pair ($\mbox{\rm cone}(L^{-}) , L^{-})$.
\end{cor}
Therefore the homotopy type depends only on $L^{-}$.

%%%%%%%%%%%%%%%%%%
\subsection{Homological local Morse Data}
Let $p$ be a \gm -critical point for $f$ with value $f(p)=c$. We define the homological local  Morse Data by
$$HMD_{\star}(p) :=H_{\star}(W_{c+\epsilon} \cap  B_{\delta}, W_{c-\epsilon} \cap  B_{\delta}) $$
 with coefficients in a field.
 
 \begin{lemma} \label{l:link-formula}
 For the homological Morse Data holds:
 $$HMD_{\star}(p) :=H_{\star}(W_{c}  \cap  B_{\delta}, W_{c} \cap  B_{\delta}-\{p\}) = \tilde{H}_{*-1}(L^{-}).$$
 %%%tilde{H}_{*-1}(lk^{-}v)
 \end{lemma}
 \begin{proof}
 By deformation reaction, see e.g. \cite{Ku}. 

  \end{proof}

In a similar way the homological transveral local Morse dat $HTMD_{\star}(p)$ and the local homological normal Morse  data $HNMD_{\star}(p)$ are defined.
Theorem \ref{t:data-product}  implies that  that the homological Morse Data is the tensor product of the corresponding homological tangential and normal Morse Data:
$$ HMD_{\star}(p) = HTMD_{\star}(p) \otimes HNMD_{\star}(p)$$

\begin{definition}
A \gm-critical point is called a \textbf{homological regular} point if $HMD_{\star}(p) = 0$. It is \textbf{homologicaly singular} if $HMD_{\star}(p) \ne  0$.\\
The total rank of $HMD_{\star}(p) $  is called the multiplicity of $p$.

We say that p is \textbf{H-critical of index k}, if  $rk HMD_k(p) \ne 0$ and this rank is referred to as the corresponding
\textbf{multiplicity} of $p$  restricted to the index k.  Compare \cite{GKR} and Definition \ref{d:hcrit}.
\end{definition}

%%%%%%%%%%%%%%%%%%%
\subsection{Morse Inequalities in stratified Morse Theory}  Kuiper  \cite{Ku} generalized in 1970  Morse inequalities to the general context of  compact metric spaces  and  continuous functions with isolated singularities. He had in mind (topological or differential) manifolds with our without boundary and simplicial complexes. Since this publication is not easy accessible\footnote{A copy can be find on on my website:
\url{https://webspace.science.uu.nl/~siers101/Kuiper-1970.pdf}}.
we will use 
Nicoalescu's  \cite{Ni}  approach with  a filtration $W_1 \subset W_2 \subset  \cdots \subset W_m = W $ of sublevel sets,
 to obtain
Morse inequalities in 
our context. The main difference with standard Morse  theory is, that we now count the \gm-critical points with multiplicities and in all dimensions where the homological Morse Data is non-vanishing. 

We introduce some notation. \\
Let 
$M_p(t) = \sum_{p,k}\mu_{p,k}(f) t^k$ where $\mu_{p,k}(f) = rk HMD_k(p)$,\\
$P_f(t) = \sum_p M_p(t) = \sum_k \mu_k(f )t^k$ where $\mu_k(f) = \sum_p \mu_{p,k}(f)$,\\
$P_W(t) = \sum_k b_k(W)$ and
$P_{W_i,W_{i-1} }(t)= \sum_k  b_k(W_i,W_{i-1})t^k $.

\begin{prop}
Morse Inequalities for \gm-critical points
$$P_f(t)  = P_W(t) + (1+t) Q( t)$$
where $Q(t)$ is a polynomial with non-negative coefficients
\end{prop}
\begin{proof}
We follow a standard proof of the Morse-inequalities, e.g. \cite{Ni}. We take a filtration with sub level sets containing only one critical point in 
the interior of $W_i \setminus W_{i-1}$. The homological Morse Data is equal to the sum off the local homological Morse Data.
Then  $P_{W_i,W_{i-1} }(p)(t)= \sum_k  b_k(HMD)(p) t ^k$.  
Next  $P_f(t) = \sum_i P_{W_i,W_{i-1}}(t)$ and continue the proof as in \cite{Ni}.

\end{proof}

\begin{cor} Weak Morse inequalities
$$ \mu_k(f)  \ge b_k(W)$$
Strong Morse inequalities;
$$ \mu_k(f) - \mu_{k-1}(f) \pm \cdots (-1)^k \mu_0(f)  \ge  b_k(W) - b_{k-1}(W) \pm \cdots (-1)^k b_0(W)   $$
Euler-Poincaré Formula:
$$\chi(W) = \sum (-1)^k \mu_k(f) $$
\end{cor}

%%%%%%%%%%%%%%%%%%%%%%%
\subsection{Stratified Morse theory of  the squared distance function $\sq_u : W \to \mathbb{R}$}

In this case we have several extra properties:
\begin{itemize}
\item  The bifurcation set allows a description with normal (hyper)planes,
\item The singularities on a stratum are only minima,
\item  The local normal Morse Data come from the linear part of $\sq_u$.
\end {itemize}

A \textbf{ bifurcation set} $\mathcal{B}$  is defined as follows. Let $\sigma$ be a 1 -dimensional simplex. Let $V_{\sigma}$ be the region between the two orthogonal hyperplanes at the the endpoints of $\sigma$. 
The bifurcation set $\mathcal{B}$ is defined as the union of the all boundary hyperplanes over all $\sigma$.
The complement of $\mathcal{B}$ divides $\mathbb{R}^N$ into chambers.

\begin{lemma}
$\sq_u$ has only Morse critical points on strata as soon as $u$ does not lie in the bifurcation set $\mathcal{B}$. Each stratified critical point is $\gm$.
\end{lemma}

\noindent
Next we determine the two types of Morse Data:

\noindent
{\bf Tangential Morse Data.} (Note that  all strata are linear).
\begin{lemma}The restriction of $\sq_u$ to a stratum, has  
\begin{itemize}
\item one interior critical point, which is clearly a local Morse minimum, or
\item  no critical point, or
\item $\sq_u$ is constant.
\end{itemize}
In case of one interior critical point:
 the Tangential Morse Data are homotopy equivalent  to the topological pair $(point, void \; set)$. 
\end{lemma}
\begin{proof}
This resembles a well known fact from high-school geometry
\end{proof}

\noindent
{\bf Normal Morse Data.} In the normal direction $\sq_u$ acts as a linear function on PL  manifold of lower dimension.  The critical point on the normal slice is the vertex $p$. The Normal Morse Data follows e.g. from the PL-description in \cite{GKR}. We have:

\begin{prop}
For the squared distance function on a PL-space
\begin{itemize}
\item the local Morse Data is homotopy equivalent to the local normal Morse Data and therefore
\item $ HMD_{\star}(p) \cong   HNMD_{\star}(p) $
\end{itemize}

\end{prop}

\begin{prop}
For the squared distance function on a PL-space: The indices  at $p$ are equal to the indices on the normal slice at $p$. An upper bound is the dimension of $W_{c-\epsilon} \cap B_{\delta}$.
\end{prop}

\subsection{Next we restrict to the PL-manifold case}  Consider  a  d-dimensional PL-manifold $W$, embedded in $\mathbb{R}^N$.  The normal slices will have dimension  $d-s$, where $s$  is the dimension of the stratum. The range is therefore $0,1,\cdots d$.\\  

In the  PL-manifold case the link of any point is a topological sphere. This also holds for the link at a normal slice.
According to Theorem 5 and its corollary, the space $L^{-}$ determines the homotopy type.  Note that $\partial L^{\pm}$ divides the link at $p$ into the regions $L^{+}$ and $L^{-}$, which are manifolds with boundary. Therefore the multiplicity in the top dimension $d-s$ is equal zero, unless we are in the case of a maximum, where the boundary is void.

In the differentiable  case we  have the standard description of a non-degenerate critical point in terms of the index, with the usual names: minimum, saddle point and maximum. In the PL-case there can occur a big variety of types. For isolated critical points the name minimum or maximum has a well-defined topological meaning. All the other points will be called (multi) saddles, described by their index and multiplicities. E.g. A money saddle has index 1 with multiplicity 2. We define the Morse vector by listing the local Morse Data; e.g. [0,2,0] for the usual monkey saddle and [0,k,0] for the more complicated saddles.
In higher dimension we we get e.g. $[0,k_1,k_2,0]$  for a mixed saddle for a 3-dimension PL-manifold.
Note that this measures only the homology; as we mentioned before: in dimension $\ge 5$ critical points can occur, where all the homologies are $0$.\\

\noindent
Note that we can also use the PL-version of lemma \ref{l:link-formula}, which is formulated in \cite{GKR}.
This uses a PL-description of the link.
\begin{equation}\label{Morse2link}
 HMD_{\star}(p) :=H_{\star}(W_{c}  \cap  B_{\delta}, W_{c} \cap  B_{\delta}-\{p\}) = \tilde{H}_{*-1}(lk^{-}v)\end{equation}
where $lk^{-}v =lk \; v \cap W_{\le}c $.\\

\begin{prop} We have
\begin{itemize}
\item  $HMD_{0}(p) = 0$ , unless we have a minimum in $p$.
\item  $HMD_{top}(p) = 0$, unless we have a maximum in $p$.
\end{itemize}
\end{prop}
\begin{proof}
In the first case $L^{-}$ is void, therefore all values from points different from $P$  are $>c$;\\
In the second case  $L^{-}$ corresponds to the whole link at $p$, therefore all values from points different from $P$  are $>c$.
\end{proof}

As we can see critical points of higher index only can occur at strata of low dimension, eg.:
\begin{prop}
Maxima of $\sq_u$ can only arise at vertices.
\end{prop}

\begin{prop}
For a PL-manifold of dimension $d$ holds the duality: $HMD_{\star}({\sq_u})(p) = HMD_{d - \star}({-\sq_u}) (p)$
\end{prop}
\begin{proof}
by Alexander Duality via $L^+$ and $L^{-}$ in a $d-1$ sphere.
\end{proof}

%%%%%%%%%%%%%%%%%%%%%%%%%%%%%%%
\section{Low dimensional results and examples}\label{lowdim}
%%%%%%%%%%%%%%%%%%%%%%%%%%%%5%%5
We describe the local  behavior of $\sq_u$ in terms of topo-regular and critical points and their types.
At vertices we use the linear approximation. In these cases  \cite {Banch}, \cite{Ba3} and  \cite{GKR} applies.

\subsection{PL-manifolds of dimension $1$}\label{ss:dim1}
We have, depending on the position of $u$:\\
\textbf{On edges}
\begin{itemize}
\item $\sq_u$ is topo-regular: no $\gm$-point,
\item there is exactly one minimum on the edge: $NMD(p)$ is void. So index $0$  with multiplicity $1$.
\end{itemize}

\textbf{On vertices}
\begin{itemize}
\item $\sq_u$ can be topo-regular, or
\item $\sq_u$  has maximum (index $1$) or minimum (index $0$) both with multiplicity $1$.
\end{itemize}

%%%ADD PIC ?? %%<<<%%%%%%%%%%%%%%%

\subsection{PL-manifolds of dimension $2$} \label{ss:dim2} 
We have, depending on the position of $u$:\\
\textbf{On 2-faces}
\begin{itemize}
\item  $\sq_u$ is topo-regular: no $\gm$- point, or
\item there is exactly one minimum on the edge: NMD is void. So index 0 with multiplicity $1$.
\end{itemize}
\textbf{On edges}
\begin{itemize}
\item $\sq_u$ is topo-regular,
\item $\sq_u$ has  a saddle (index $1$) or a minimum (index $0$) both of multiplicity $1$ (normal data is 1-dimensional),
\end{itemize}
This follows from the normal slice and the  1-dimensional case above.
\textbf{On vertices} (see Proposition 6.1 of \cite{GKR}):
\begin{itemize}
\item local minimum (index $0$, multiplicity $1$,
\item saddle point (index 1, multiplicity arbitrary)
\item local maximum (index $2$, multiplicity $1$,
\end{itemize}
We list the local Morse Data in the left hand  side of Table 1.

\begin{figure}[h]

\includegraphics[width=12cm]{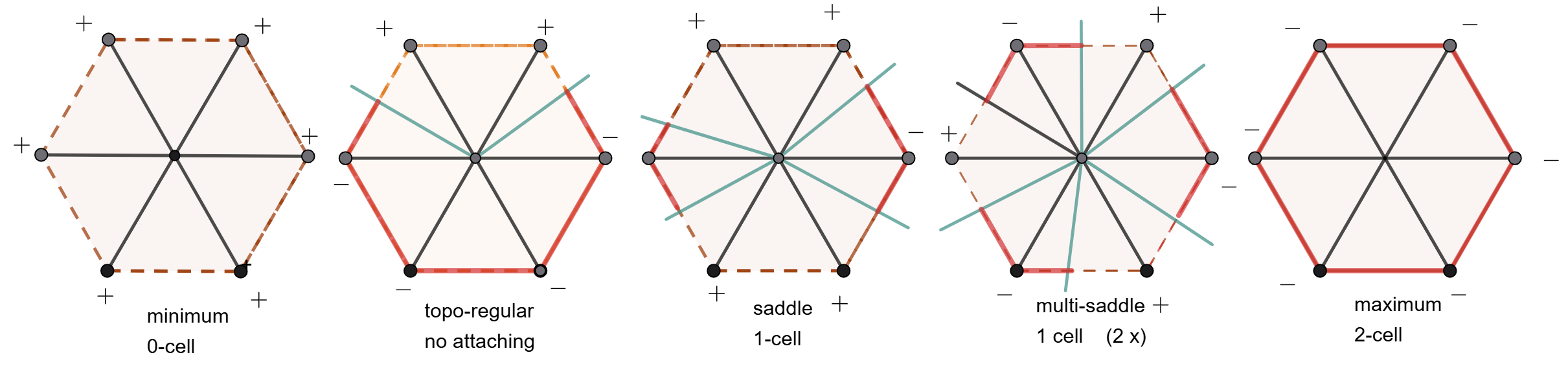}
\caption{Several types in dimension 2} \label{fig:dim2types}
\end{figure}
About Figure \ref{fig:dim2types} : the pictures show a star-neighborhood of $v$. We use formula (\ref{Morse2link}) and mark $lk^{-}v $. 
A neighborhood of  $v$  is via a PL-map transformed in a planar picture. We took hexagons, but similar pictures exist for arbitrary k-gons around $v$.

\begin{table}\label{m-table}
\parbox{.40\linewidth}{
\centering
\begin{tabular}{| c | c | c| c |}
\hline
\text{dim} & f       & e & v \\ 
\hline
 2             &         &           & 0 \; 0 \; 1 \\
 \hline
1             &           & 0 \; 1 & 0 \; k \; 0 \\
\hline
0            &     1    & 1\; 0   &  1 \; 0 \; 0 \\
\hline
\end{tabular}
%\caption{Foo}
}
\hfill
\parbox{.45\linewidth}{
\centering
\begin{tabular}{| c | c | c| c |c |}
\hline
\text{dim} & f\textsubscript{3}    &  f\textsubscript{2}      & e & v \\ 
\hline
3              &           &           &                &  0 \; \; 0 \; \; 1 \\
\hline     
 2             &         &           & 0 \; 0 \; 1   & 0 \;  \; k\textsubscript{2} \; 0 \\
 \hline
1             &           & 0 \; 1 & 0 \; k \; 0  & 0 \;  \; k\textsubscript{1} \; 0\\
\hline
0            &     1    & 1\; 0   &  1 \; 0 \; 0 &  1 \; \; 0 \; \; 0 \\
\hline
\end{tabular}
}
\vspace{0.2cm}
\caption{Morse-vectors in dimension 2 and 3.\\ 
\centering{\small{v= vertices, e= edges, f = faces}}}
\end{table}

\iffalse%%%%%%%%%%%%%BANCHOFF EXAMPLE<<<<
\begin{figure}
\subsection{Banchoff Example}
In Figure REF  shew  two pyramids  are joint over a common quadrangle. This is taken frm \cite{Banch}??.
WE consider the quadratic function $\sq_u$ with $u$ a point on the diagonal $vw$ near to $w$. We detect the following critical points:\\
3 Maxima at  vertices $u_1,u_2,w$,\\
7 saddles at edges: $q_1u_2,q_1u_1,q_1w,q_2u_2,q_2u_1,q_2w$ and at vertex $v$\\
6 minima  at faces $u_1q_1w,u_1wq_2,u_1q_2v,u_2q_2w,u_2q_w$.
CHECK\\
If we take a linear function (vertical from below), then we find\\
1 minimum at $v$, 2 maxima at $u_1,u_2$, 1 saddle at $v$.  $\chi = 1-1+2 =2$.

\includegraphics[width=12cm]{baanchoff-pic}
\caption{Banchoff-example} \label{fig:banchoff-pic}
\end{figure}
\fi%%%%%%%%%%%%%%%END-BANCHOFF<<<<<<<<

\subsection{PL-manifolds of dimension 3}
The Morse vectors for $\sq_u$ are depicted  at the right side of Table  1. %%%\ref{m-table}.XXXXXX
For vertices they are consequence of Theorem 6.3 of \cite{GKR}. For other dimensional strata we use the normal Morse Data.  They can be copied from the  1- and 2-dimensional manifold as written before in subsections \ref{ss:dim1} and \ref{ss:dim2}.
Examples of mixed  critical points; e.g. with Morse vector $[0; k_1 ; k_2; 0]$, for linear functions are given in \cite{GKR} and they can also serve in the case of the squared distance.
Other examples can be constructed via suspension. Given a PL-manifold $W \subset \mathbb{R}^N$ construct $\Sigma W$ by  suspension with two points in $\mathbb{R}^{N+1}$ then e.g the multi saddle of index 1 in the surface case can induce a multi saddle in the normal direction of an edge.

\subsection{PL-manifolds of dimension $ \ge 4$}
Critical points  in dimension 4 can be described, in a similar way;  we leave this to the reader. From dimension $5$ on there is the complication, that homologically regular is no longer the same as topo-regular. We refer to \cite{GKR} for more information and examples.\\

\subsection{ Manifolds with boundary} It is again sufficient to work with linear functions normal slices. For linear functions see  Section 5 of \cite{GKR}.  So for $\sq_u$  a good description of critical points is available  and $HMD(p)$  can be computed via the link-formula (\ref{Morse2link}). Note that if $u$ is an internal punt of $W$, then it is a (absolute) minimum.\\
Note that im several case critical points on $\partial W$ give rise to topo-regular points, eg for a top-dimension boundary face  there are only two options: (1) a minimum if u lies in an outside direction; and (2)  topo-regular  if $u$ lies in an inside direction.

%%%%%%%%%%%%%%%%%%
\section{Counting normals}\label{s:count}
A  normal from a point to a manifold $W$  corresponds in the smooth setting precisely to a critical point of $\sq_u$. We extend this correspondence to the PL-case, where we take the critical point criterion as definition.  In some cases this can look a little bit strange, especially in points where $\sq_u$ is homologically regular but not topologically regular.

The counting of  normals has a  history. As an example:
It is conjectured since long that for any convex body  $W \subset R^N$  there exists
a point in the interior of W which belongs to at least 2N normals from different points
on the boundary of W. The conjecture is known to be true for N = 2,3,4. See the work of Heil \cite{Heil} and Pardon \cite{Pa}.

There are several results in the smooth case:
Recently Y. Martinez-Maure proved in \cite{Ma}  for n = 3,4 that (under some mild conditions) almost every normal through a boundary point to a smooth convex body W passes
arbitrarily close to the set of points lying on normals through at least six distinct points
of  $\partial W$. This result was reproved by Grebennikov and Panina in \cite{GP}. 
Extending their method Panina and Siersma proved related results for concurrent normals of immersed smooth compact
manifolds. See \cite{PS}.

\noindent
We mention the following results in the PL-case:
\begin{itemize}

\item  For any convex polytope $P \subset \mathbb{R}^N$  there  is always a point in $P$ with at least $2N+2 $ normals to $\partial P$ \cite{NPS},
\item  In general: the total number of normals is less or equal to the  total number of faces \cite{NPS},
\item Each simple polytope  has an inside  point with $2N+4$  normals to the boundary \cite{NP1} (in dimension $\ge 3$),
\item Each generic polytope in $\mathbb{R}^3$ has an inside point with ten normals to the boundary \cite{NP2},
\item For a generic tetrahedron in $\mathbb{R}^3$ the  total maximum number of normals can only be 10, 12 or 14, cf \cite{NP2},
\item For every generic closed piecewise linear curve in $\mathbb{R}^3$ there is a point with at least 8 emanating normals to the curve. If the curve is knotted, there is a point with at least 10 emanating normals  \cite{PS-bat} .
\end{itemize}

\section{Towards a Discrete Morse function}\label{s:dmt}
We assume that the reader is familiar with the several notions from Discrete Morse Theory  (DM-theory). We refer to Forman  \cite{Fo}.  For the equivalence between critical points in PL-theory and DM-theory we refer to \cite{FLV}, which is an essential background for this section, see also \cite{ES}.

It seems rather evident, that the critical points of $\sq_u$ are related to DM-theory: just look to the position of critical points. Assume $W$ is a PL-manifold with the structure of a  simplicial complex. Critical points of $\sq_u$ occur on `selected' simplices. Besides the usual (Morse) maxima and minima also multi-saddles can occur.
%Near a vertex $\sq_u$ can be approximated by its linear approximation and we get locally  the context of PL-manifold with a linear function. 
Note that the index of the $\sq_u$-critical points does not fit well with the dimension of the simplex. E.g. on a top-dimensional simplex a critical point has index 0, and not the top-dimension. 
 In stead of $\sq_u$  we will work with $-\sq_u$ or with $c_u - \sq_u$,  where $c_u$ be the maximum value of $\sq_u$ on $W$.  They have the same critical points as $\sq_u$, but `opposite' indices. Now a top simplex can contain only a maximum  as critical point.  

We intend to construct for every $u \in \mathbb{R}^N$ a (companion) discrete Morse function $H_u$ on the simplicial complex,  which share a number of properties with $\sq_u$; such as a correspondence of critical points points).  For a similar study in case of Voronoi-distance, see \cite{SvM}.

For a PL-manifold  and a linear function a correspondence is known:

\begin{prop}\label{p:corres}(\cite{FLV}, Cor 11)
 Let $d \le 3$ and $f$ a PL-function.
There exist a discrete Morse function $\hat{f}$ and a $1 \leftrightarrow k_i$ correspondence between PL-critical points of $f$ of index $i$ and multiplicity $k_i$ and discrete critical $i$ simplices  of $\hat{f}$.\\
Each discrete simplex belongs to the star of the critical point. % (EXPLAIN STAR)\\
\end{prop}

We want a similar statement to be true for the quadratic distance function.
We sketch two approaches:

{\it 1st approach}:
Approximate $\sq_u$ by a PL-function. \\
Let $K$ be the underlying simplicial complex of the PL-manifold $W$.
We want to subdivide $K$, containing the critical points of $\sq_u$ as vertices and  such  that there exists a PL-map $g: W \to \mathbb{R}$, such that $g$ and $\sq_u$ are topologically equivalent. Probably this can be done for $d$ small. It requires anyhow a result in PL-topology.
Assuming this, apply directly the proposition \ref{p:corres}  and this gives a DM-function in the correspondence.

{\it 2d approach}: Updating the proof of the PL-statement. \\
%Assume (if necessary after subdividing) that (closed) stars of critical simplices (i.e simplices containing a critical point) are disjoint.
We  follow the reasoning in \cite{FLV} and start with defining a discrete function on (the cells of) $K$ by   $h_u(\sigma) = \max_{y \in \sigma} (c_u -\sq_u(y) )$. This will not be  a DM-function, since e.g. equalities  $h_u(\sigma) = h_u(\tau)$ if $\sigma = \tau$ can occur  already in the case vertex-edge (as soon as the edge does not contain a GM-critical point) and also for higher dimensional simplices for similar reasons.

Let $K_{\le c}$ be the the set of cells where $h_u{\sigma} \le c$. Let $c'$ be the predecessor of $c$ ( or $c-1$ in case of no predecessor).The set of cells $K_{\le c} \setminus  K_{\le c'}$ consists of cells with value exactly $c$.

We start to approximate $h_u$ by a discrete Morse function $\hat{h}_u$ as in \cite{FLV}
 Theorem 6 and Corollary  7.
If a critical point of $\sq_u$ is a vertex of $K$ then we follow the modifications of \cite{FLV} and construct a discrete gradient field near $v$ with the appropriate critical cells. If the critical point is a  simplex we have to modify this construction (e.g. via the normal slice); if the critical point lies in a top dimensional simplex, then we take this as a critical simplex in the sense of DM-theory.
The methods from \cite{FLV} will allow us to join the local contributions to a global discrete gradient field and a DM-function.\\
%%%%%%%%%%%%%%%%%%%%%%%
%END OF THE STORY
%%%%%%%%%%%%%%%%%%%%%%%%%

\section{Tools}
No AI and other software was used in this work. The Figure has been produced by the author himself using Geogabra.

%%%%%%%%%%%%%%%%%%%%%%%%%%%%%%%%%%%%%%%%%%%%%

\end{document}